\documentclass[conference]{IEEEtran}
\IEEEoverridecommandlockouts 

\usepackage{cite}
\usepackage{amsmath,amssymb,amsfonts}
\usepackage{algorithmic}
\usepackage{graphicx}
\usepackage{textcomp}
\usepackage{xcolor}
\usepackage{booktabs}
\usepackage{amsmath}
\usepackage{amssymb}
\usepackage{amsthm}
\usepackage[hidelinks]{hyperref} 
\usepackage{balance}
\usepackage{etoolbox} 
\theoremstyle{plain}
\newtheorem{theorem}{Theorem}

\theoremstyle{plain}
\newtheorem{lemma}{Lemma}

\makeatletter
\newif\if@printsuffix \@printsuffixtrue
\newif\if@firstpara   \@firstparatrue
\let\orig@par\par
\def\par{%
  \ifhmode
    \if@printsuffix
      \if@firstpara
        \global\@firstparafalse
      \else
        \space//%
      \fi
    \fi
  \fi
  \orig@par
}
\pretocmd{\section}{\global\@firstparatrue}{}{}
\pretocmd{\subsection}{\global\@firstparatrue}{}{}
\pretocmd{\subsubsection}{\global\@firstparatrue}{}{}
\pretocmd{\paragraph}{\global\@firstparatrue}{}{}
\pretocmd{\subparagraph}{\global\@firstparatrue}{}{}
\AtBeginEnvironment{abstract}{\@printsuffixfalse \global\@firstparatrue}%
\AtEndEnvironment{abstract}{\@printsuffixtrue}%
\AtBeginEnvironment{IEEEkeywords}{\@printsuffixfalse \global\@firstparatrue \setlength{\parskip}{0pt}}%
\AtEndEnvironment{IEEEkeywords}{\@printsuffixtrue}%

\newlength{\BodyParskip}
\newif\if@afterheading
\@afterheadingfalse
\pretocmd{\section}{\global\@afterheadingtrue \parskip 0pt}{}{}
\pretocmd{\subsection}{\global\@afterheadingtrue \parskip 0pt}{}{}
\pretocmd{\subsubsection}{\global\@afterheadingtrue \parskip 0pt}{}{}
\pretocmd{\paragraph}{\global\@afterheadingtrue \parskip 0pt}{}{}
\pretocmd{\subparagraph}{\global\@afterheadingtrue \parskip 0pt}{}{}
\everypar{\if@afterheading \global\@afterheadingfalse \parskip\BodyParskip \fi \ignorespaces}

\makeatother

\title{On a group of invariances in a class of functions}

\author{%
  \IEEEauthorblockN{Shravan Mohan}
  \IEEEauthorblockA{17-004, Mantri Residency, Bangalore}
}

\begin{document}

\maketitle

\begingroup
\setlength{\parskip}{0pt}
\begin{abstract}
A class of parametric functions formed by alternating compositions of multivariate polynomials and rectification‑style monomial maps of the form $(x)_+^{\alpha}$ is studied (the layer-wise exponents $\alpha$ are treated as fixed hyperparameters and are not optimized). For this family, nontrivial parametric invariances are identified and characterized: distinct parameter settings that induce identical input–output maps. A constructive description of the invariance structure is provided, enabling sparse function representations, parameter obfuscation, and potential dimensionality reduction for optimization.
\end{abstract}
\endgroup


\section{Introduction}
\noindent Parametric maps obtained by alternating multivariate polynomials and componentwise monomial (rectification‑style) maps are studied. Formally, let \(X \subseteq \mathbb{R}^{d_0}\) and consider maps \(f_{\theta}: X \to \mathbb{R}^{d_L}\) of the form
\[
  f_{\theta} \;=\; P_{L} \circ M_{L-1} \circ P_{L-1} \circ \cdots \circ M_{1} \circ P_{1},
\]
where for each layer $\ell$
\begin{itemize}
  \item \(P_{\ell}:\mathbb{R}^{d_{\ell-1}}\to\mathbb{R}^{d_{\ell}}\) is a multivariate polynomial represented as a sum of powers of affine linear forms per output. Concretely, for each output \(j=1,\dots,d_{\ell}\), fix an integer \(K_{\ell,j}\ge 1\) and degrees \(r^{(\ell)}_{j,k}\in\mathbb{N}_{0}\) (treated as hyperparameters). With output-specific affine forms \((w^{(\ell)}_{j,k}, b^{(\ell)}_{j,k})\in\mathbb{R}^{d_{\ell-1}}\times\mathbb{R}\),
    \[
      [P_{\ell}(z)]_{j} \;=\; \sum_{k=1}^{K_{\ell,j}} \Big( (w^{(\ell)}_{j,k})^{\top}z + b^{(\ell)}_{j,k} \Big)^{\,r^{(\ell)}_{j,k}}.
    \]
  \item \(M_{\ell}:\mathbb{R}^{d_{\ell}}\to\mathbb{R}^{d_{\ell}}\) is a componentwise rectified monomial map with a layer-wise exponent \(\alpha_{\ell}>0\) (independent of the element index):
    \[
      [M_{\ell}(z)]_{i} \;=\; (z_{i})_+^{\alpha_{\ell}}.
    \]
\end{itemize}

\noindent The representation of \(P_{\ell}\) as a sum of powers of affine forms is central to the framework. By the classical Waring decomposition, every homogeneous polynomial of degree \(d\) in \(n\) variables can be expressed as a sum of \(d\)-th powers of linear forms; the minimal number of such terms required is called the \emph{Waring rank}. For inhomogeneous polynomials, homogenization yields an analogous decomposition in terms of affine forms. Here, each output component \([P_{\ell}(z)]_{j}\) is represented as \(\sum_{k=1}^{K_{\ell,j}} ((w^{(\ell)}_{j,k})^{\top}z + b^{(\ell)}_{j,k})^{r^{(\ell)}_{j,k}}\), where distinct terms may have different degrees \(r^{(\ell)}_{j,k}\). This generalizes the classical setting: allowing multiple degrees per output increases expressiveness while preserving the invariance structure explored in this paper. The number of terms \(K_{\ell,j}\) is treated as a hyperparameter; for efficiency, one may choose \(K_{\ell,j}\) close to a known or estimated Waring rank, though larger values permit additional modeling flexibility.\\

\noindent\textbf{Example (two-layer quadratic composition).} Take \(d_{0}=3\), \(d_{1}=2\), \(d_{2}=1\) and \(\alpha_{1}=2\). Let the first polynomial \(P_{1}:\mathbb{R}^{3}\to\mathbb{R}^{2}\) have outputs
\[
  [P_{1}(x)]_{1} \;=\; \big((w^{(1)}_{1,1})^{\top}x + b^{(1)}_{1,1}\big)^{2} \;+\; \big((w^{(1)}_{1,2})^{\top}x + b^{(1)}_{1,2}\big),
\]
\[
  [P_{1}(x)]_{2} \;=\; \big((w^{(1)}_{2,1})^{\top}x + b^{(1)}_{2,1}\big)^{2} \;+\; \big((w^{(1)}_{2,2})^{\top}x + b^{(1)}_{2,2}\big),
\]
with \(w^{(1)}_{i,k}\in\mathbb{R}^{3}\), \(b^{(1)}_{i,k}\in\mathbb{R}\) for \(i\in\{1,2\},k\in\{1,2\}\). The monomial layer \(M_{1}\) has components \([M_{1}(z)]_{i}=(z_{i})_{+}^{2}\). Let the second polynomial \(P_{2}:\mathbb{R}^{2}\to\mathbb{R}\) be
\[
  P_{2}(z) \;=\; \big((w^{(2)}_{1,1})^{\top}z + b^{(2)}_{1,1}\big)^{2} \;+\; \big((w^{(2)}_{1,2})^{\top}z + b^{(2)}_{1,2}\big),
\]
with \(w^{(2)}_{1,k}\in\mathbb{R}^{2}\), \(b^{(2)}_{1,k}\in\mathbb{R}\) for \(k\in\{1,2\}\). The full map is
\[
  f_{\theta}(x) \;=\; P_{2}\!\big(M_{1}(P_{1}(x))\big).
\]
This matches the template with: for \(P_{1}\), \(K_{1,1}=K_{1,2}=2\) and degrees \((2,1)\) using \(\big(w^{(1)}_{1,1}, b^{(1)}_{1,1}\big), \big(w^{(1)}_{1,2}, b^{(1)}_{1,2}\big)\) for output 1 and \(\big(w^{(1)}_{2,1}, b^{(1)}_{2,1}\big), \big(w^{(1)}_{2,2}, b^{(1)}_{2,2}\big)\) for output 2; for \(P_{2}\), \(K_{2,1}=2\) with degrees \((2,1)\) using \(\big(w^{(2)}_{1,1}, b^{(2)}_{1,1}\big)\) and \(\big(w^{(2)}_{1,2}, b^{(2)}_{1,2}\big)\).\\

\noindent All free parameters are collected into \(\theta\) (the polynomial affine parameters \(\{w^{(\ell)}_{j,k},\,b^{(\ell)}_{j,k}\}_{\ell,j,k}\)). The layer-wise exponents \(\alpha_{\ell}\) are treated as fixed hyperparameters and are not included in \(\theta\). The central problem of this paper is to characterize the parametric invariances of this family:
\[
  \theta \sim \theta' \quad\Longleftrightarrow\quad f_{\theta}(x)=f_{\theta'}(x)\ \text{for all}\ x\in X.
\]
Equivalently, transformations \(T:\Theta\to\Theta\) (a group or semigroup of reparameterizations) are sought such that
\[
  f_{T(\theta)}\equiv f_{\theta}.
\]
Understanding this invariance relation determines identifiability, effective degrees of freedom, and canonical representatives for parsimonious modeling. Several papers in the past have looked into invariances in neural networks which are a special case of the present work \cite{entezari2021role, le2022training, simsek2021geometry, song2018accelerating, zhao2022symmetries}. \\

\noindent Three practical consequences are associated with parametric invariances in compositional models: parameter ranges can be minimized within equivalence classes, enabling tighter dynamic ranges and improved low‑precision discretization (the realized map is equivalent before discretization; after discretization the chosen representative typically incurs smaller quantization error at a fixed precision); obfuscation is permitted since parameters can be reparameterized and inputs can be transformed so that structure is hidden while observable behavior is preserved; and invariance-aware optimization with regularization: when the objective includes regularizers (e.g., $\ell_2$, norm-balancing, sparsity), a short subroutine can be inserted at each gradient step to reparameterize within the equivalence class and (approximately) minimize the regularization term while keeping the data-fit unchanged, which improves conditioning and can accelerate convergence.

\section{The Invariance Group}
\noindent Element-wise rectified monomials are equivariant to permutations and positive diagonal scalings in a simple way. Fix a layer \(\ell\) and define \(\sigma_{\alpha_{\ell}}:\mathbb{R}^{d_{\ell}}\to\mathbb{R}^{d_{\ell}}\) by \([\sigma_{\alpha_{\ell}}(z)]_{i}=(z_{i})_{+}^{\alpha_{\ell}}\). Let \(P\in\mathbb{R}^{d_{\ell}\times d_{\ell}}\) be a permutation matrix and \(D=\mathrm{diag}(d_{1},\dots,d_{d_{\ell}})\) a diagonal matrix with \(d_{i}>0\). Write \(D^{\alpha_{\ell}}:=\mathrm{diag}(d_{1}^{\alpha_{\ell}},\dots,d_{d_{\ell}}^{\alpha_{\ell}})\). Then, for all \(x\in\mathbb{R}^{d_{\ell}}\),
\[
  \sigma_{\alpha_{\ell}}(P D x)\;=\;P\, D^{\alpha_{\ell}}\, \sigma_{\alpha_{\ell}}(x).
\]
Thus, the group \(G_{\ell}:=\mathcal{S}_{d_{\ell}}\ltimes (\mathbb{R}_{>0}^{d_{\ell}},\cdot)\) acts on inputs by \(x\mapsto P D x\) and on outputs by \(y\mapsto P D^{\alpha_{\ell}} y\), under which \(\sigma_{\alpha_{\ell}}\) is equivariant.\\

\noindent This identity follows by a component-wise calculation: for index \(i\) with permutation \(p(i)\),
\[
\begin{aligned}
  [\sigma_{\alpha_{\ell}}(P D x)]_{i}
  &= \big( d_{p(i)}\, x_{p(i)} \big)_{+}^{\alpha_{\ell}}
   = d_{p(i)}^{\alpha_{\ell}}\, (x_{p(i)})_{+}^{\alpha_{\ell}} \\
  &= \big[P\, D^{\alpha_{\ell}}\, \sigma_{\alpha_{\ell}}(x)\big]_{i}.
\end{aligned}
\]
Consequences for the full composition follow by interleaving polynomial layers: permutations and positive diagonal scalings can be pushed through rectified monomial layers as \(D\mapsto D^{\alpha_{\ell}}\), and absorbed into adjacent polynomial maps as parameter reparameterizations. This underlies the diagonal/permutation invariances used elsewhere in the paper.\\

\noindent A complete layerwise invariance can be described starting from \(P_{1}\) upward. Let the overall map be
\[
  f_{\theta} \;=\; P_{L}\circ M_{L-1}\circ P_{L-1}\circ \cdots \circ M_{1}\circ P_{1}.
\]
Two operations generate an invariance family.
\begin{enumerate}
\item \textit{Input linear reparameterization}: For any invertible \(S_{0}\in \mathrm{GL}(d_{0})\), define a new first polynomial
\[
  \tilde P_{1}(z)\;:=\;P_{1}(S_{0}^{-1} z).
\]
Then for all \(x\), the jointly transformed pair satisfies
\[
  \tilde f_{\tilde\theta}(S_{0}x)\;=\;f_{\theta}(x),
\]
i.e., a linear change of coordinates at the input can be compensated by the inverse precomposition into \(P_{1}\).

\item \textit{Inter-layer permutation/diagonal reparameterization}: For each hidden interface \(\ell=1,\dots,L-1\), choose a permutation \(\Pi_{\ell}\in\mathcal{S}_{d_{\ell}}\) and a positive diagonal \(D_{\ell}=\mathrm{diag}(d_{\ell,1},\dots,d_{\ell,d_{\ell}})\). Define reparameterized polynomials
\[
\begin{aligned}
  \tilde P_{1} &:= (\Pi_{1} D_{1}^{\,1/\alpha_{1}})\circ P_{1},\\
  \tilde P_{\ell} &:= (\Pi_{\ell} D_{\ell}^{\,1/\alpha_{\ell}})\circ P_{\ell}\circ (\Pi_{\ell-1} D_{\ell-1})^{-1},\\[-2pt]
                 &\hspace{1.6em}\ell=2,\dots, L-1,\\
  \tilde P_{L} &:= P_{L}\circ (\Pi_{L-1} D_{L-1})^{-1},
\end{aligned}
\]
and keep all monomial layers unchanged, \(\tilde M_{\ell}=M_{\ell}\). Using the equivariance
\[
  M_{\ell}\!\big(\Pi_{\ell} D_{\ell}^{\,1/\alpha_{\ell}}\, z\big)\;=\;\Pi_{\ell} D_{\ell}\, M_{\ell}(z),
\]
one checks by cancellation that
\[
  \tilde f_{\tilde\theta}(x) \;=\; f_{\theta}(x)\quad \text{for all }x.
\]
\end{enumerate}
Thus inserting \(\Pi_{\ell} D_{\ell}^{\,1/\alpha_{\ell}}\) after \(P_{\ell}\) forces the post-\(M_{\ell}\) activation to be multiplied by \(\Pi_{\ell} D_{\ell}\); compensating by precomposing \(P_{\ell+1}\) with \((\Pi_{\ell} D_{\ell})^{-1}\) leaves the overall input–output map unchanged. Combined with the input reparameterization above, this yields a complete, layerwise description of the invariance group for the composed model.\\

\noindent To understand how the permutations and diagonal scalings get absorbed into the polynomial parameters, we adopt an ordered-parameter view: for each layer \(\ell\), collect its per‑output parameters into an ordered list
\[
  \phi^{(\ell)} \;=\; \big(\,\phi^{(\ell)}_{1},\,\dots,\,\phi^{(\ell)}_{d_{\ell}}\,\big),
  \quad
  \phi^{(\ell)}_{j} \;=\; \big\{\, (w^{(\ell)}_{j,k},\,b^{(\ell)}_{j,k},\,r^{(\ell)}_{j,k}) \,\big\}_{k=1}^{K_{\ell,j}},
\]
where the degrees \(r^{(\ell)}_{j,k}\) are hyperparameters that move with their element under permutations but are not numerically changed. Firstly, an input reparameterization \(x\mapsto S_{0}x\) (write \(A_{0}:=S_{0}^{-1}\)) sends
\[
(w^{(1)}_{j,k},\,b^{(1)}_{j,k})\mapsto (A_{0}^{\!\top} w^{(1)}_{j,k},\, b^{(1)}_{j,k})
\]
for all \(j,k\). Secondly, applying a post‑map at \(P_{\ell}\) by \(\Pi_{\ell} D_{\ell}^{\,1/\alpha_{\ell}}\) (permutation plus diagonal scaling) with \(s_{\ell}:=\mathrm{diag}(D_{\ell}^{\,1/\alpha_{\ell}})\) and \(j':=\Pi_{\ell}^{-1}(j)\) yields
\(
  \tilde\phi^{(\ell)}_{j} = \big\{\, (\tilde w^{(\ell)}_{j,k},\,\tilde b^{(\ell)}_{j,k},\,\tilde r^{(\ell)}_{j,k}) \,\big\}_{k=1}^{K_{\ell,j}},
\)
where,
\[
  \tilde w^{(\ell)}_{j,k} = s_{\ell,j}\, w^{(\ell)}_{j',k},\quad
  \tilde b^{(\ell)}_{j,k} = s_{\ell,j}\, b^{(\ell)}_{j',k},\quad
  \tilde r^{(\ell)}_{j,k} = r^{(\ell)}_{j',k}.
\]
In other words, the list is reordered by \(\Pi_{\ell}\) and each output coordinate is scaled by \(s_{\ell,j}\). Finally, precomposing \(P_{\ell+1}\) by \((\Pi_{\ell} D_{\ell})^{-1}\) (write \(B_{\ell}:=(\Pi_{\ell} D_{\ell})^{-1}\)) updates every incoming affine form as
\[
  (w^{(\ell+1)}_{j,k},\,b^{(\ell+1)}_{j,k}) \;\mapsto\; (B_{\ell}^{\!\top} w^{(\ell+1)}_{j,k},\, b^{(\ell+1)}_{j,k}),
\]
and the ordered list \(\phi^{(\ell+1)}\) is relabeled by the same \(\Pi_{\ell}\) so that interface indices remain aligned. Degrees again only permute with their element.\\

\noindent In summary, at each layer, the ordered list is first permuted and diagonally scaled at the output of \(P_{\ell}\) via \(\Pi_{\ell} D_{\ell}^{\,1/\alpha_{\ell}}\); the subsequent layer \(P_{\ell+1}\) is precomposed by \(B_{\ell}=(\Pi_{\ell} D_{\ell})^{-1}\) and its ordered list is permuted by the same \(\Pi_{\ell}\). Input reparameterization \(x\mapsto S_{0}x\) induces the left‑multiplication \(A_{0}^{\!\top}\) of \(w\) in \(P_{1}\). Throughout, degrees \(r\) are preserved (only reordered).\\

\noindent\textbf{Continuing the example (explicit reparameterization).} Keep \(d_{0}=3\), \(d_{1}=2\), \(d_{2}=1\), \(\alpha_{1}=2\). Choose:
- an input change \(S_{0}\in \mathrm{GL}(3)\) and let \(A_{0}:=S_{0}^{-1}\);
- a permutation \(\Pi_{1}=\begin{bmatrix}0&1\\[2pt]1&0\end{bmatrix}\) (swap the two hidden coordinates);
- a positive diagonal \(D_{1}=\mathrm{diag}(d_{1},d_{2})\) with \(d_{1},d_{2}>0\), and set \(s_{1}:=D_{1}^{1/\alpha_{1}}=\mathrm{diag}(\sqrt{d_{1}},\sqrt{d_{2}})\). Apply the three updates to the example parameters.\\

\noindent 1) Input reparameterization into \(P_{1}\) (left-multiply each weight by \(A_{0}^{\!\top}\); biases unchanged):
\[
\begin{aligned}
  &w^{(1)}_{1,1}\leftarrow A_{0}^{\!\top} w^{(1)}_{1,1},\quad w^{(1)}_{1,2}\leftarrow A_{0}^{\!\top} w^{(1)}_{1,2},\\
  &w^{(1)}_{2,1}\leftarrow A_{0}^{\!\top} w^{(1)}_{2,1},\quad w^{(1)}_{2,2}\leftarrow A_{0}^{\!\top} w^{(1)}_{2,2}.
\end{aligned}
\]\\

\noindent 2) Post‑map at \(P_{1}\) by \(\Pi_{1} s_{1}\) (permute outputs and scale with \(\sqrt{d_{1}},\sqrt{d_{2}}\)). The new output-1 parameters come from old output-2 scaled by \(\sqrt{d_{1}}\); new output-2 from old output-1 scaled by \(\sqrt{d_{2}}\):
\[
\begin{aligned}
  &\tilde w^{(1)}_{1,1}=\sqrt{d_{1}}\; w^{(1)}_{2,1},\quad \tilde b^{(1)}_{1,1}=\sqrt{d_{1}}\; b^{(1)}_{2,1},\\
  &\tilde w^{(1)}_{1,2}=\sqrt{d_{1}}\; w^{(1)}_{2,2},\quad \tilde b^{(1)}_{1,2}=\sqrt{d_{1}}\; b^{(1)}_{2,2},\\[4pt]
  &\tilde w^{(1)}_{2,1}=\sqrt{d_{2}}\; w^{(1)}_{1,1},\quad \tilde b^{(1)}_{2,1}=\sqrt{d_{2}}\; b^{(1)}_{1,1},\\
  &\tilde w^{(1)}_{2,2}=\sqrt{d_{2}}\; w^{(1)}_{1,2},\quad \tilde b^{(1)}_{2,2}=\sqrt{d_{2}}\; b^{(1)}_{1,2}.
\end{aligned}
\]

3) Precompose \(P_{2}\) by \(B_{1}:=(\Pi_{1} D_{1})^{-1}\). Since \(B_{1}^{\!\top}=\Pi_{1} D_{1}^{-1}\), for each \(k\in\{1,2\}\) write \(w^{(2)}_{1,k}=\begin{bmatrix}w^{(2)}_{1,k,1}\\ w^{(2)}_{1,k,2}\end{bmatrix}\). Then
\[
  \tilde w^{(2)}_{1,k} \;=\; \Pi_{1} D_{1}^{-1}\, w^{(2)}_{1,k}
  \;=\;
  \begin{bmatrix}
    \;\dfrac{w^{(2)}_{1,k,2}}{d_{2}}\;\\[6pt]
    \;\dfrac{w^{(2)}_{1,k,1}}{d_{1}}\;
  \end{bmatrix},
  \qquad
  \tilde b^{(2)}_{1,k} \;=\; b^{(2)}_{1,k}.
\]
The monomial layer \(M_{1}\) is unchanged. With these updates, the reparameterized model \(\tilde f\) satisfies \(\tilde f(x)=f(x)\) for all \(x\).

\section{Invariance Aware Optimization }

\noindent\textbf{Step 1 (first-layer Frobenius cost only; fix \(S_{0}=I\)).}
Let the first-layer regularizer be
\[
  \mathcal{R}^{(1)}_{\mathrm{F}}\;=\;\sum_{j=1}^{d_{1}}\sum_{k=1}^{K_{1,j}}\Big(\|w^{(1)}_{j,k}\|_{2}^{2}+\mu\,(b^{(1)}_{j,k})^{2}\Big),\quad \mu\ge 0.
\]
In this step we fix \(S_{0}=I\) (so \(A_{0}:=S_{0}^{-1}=I\)) and do not optimize over it. We only apply an output scaling at layer 1 by \(D_{1}=\mathrm{diag}(d_{1,1},\dots,d_{1,d_{1}})\) via \(s_{1}:=D_{1}^{1/\alpha_{1}}\). For each output \(j\) and term \(k\) with degree \(r^{(1)}_{j,k}\),
\[
  w^{(1)}_{j,k}\ \mapsto\ s_{1,j}^{\,1/r^{(1)}_{j,k}}\, w^{(1)}_{j,k},\qquad
  b^{(1)}_{j,k}\ \mapsto\ s_{1,j}^{\,1/r^{(1)}_{j,k}}\, b^{(1)}_{j,k}.
\]
Therefore the transformed first-layer cost is
\[
\begin{aligned}
  \mathcal{R}^{(1)}_{\mathrm{F}}(D_{1})
  &= \sum_{j=1}^{d_{1}}\sum_{k=1}^{K_{1,j}}
     \Big(\,\|w^{(1)}_{j,k}\|_{2}^{2}+\mu\,(b^{(1)}_{j,k})^{2}\,\Big) s_{1,j}^{\,2/r^{(1)}_{j,k}}\\
  &= \sum_{j=1}^{d_{1}}\sum_{k=1}^{K_{1,j}}
     \underbrace{\Big(\,\|w^{(1)}_{j,k}\|_{2}^{2}+\mu\,(b^{(1)}_{j,k})^{2}\,\Big)}_{=:a^{(1)}_{j,k}}\;
     d_{1,j}^{\,\gamma^{(1)}_{j,k}},
\end{aligned}
\]
with exponents \(\gamma^{(1)}_{j,k}:=\tfrac{2}{\alpha_{1} r^{(1)}_{j,k}}\). Permutations at layer 1 do not affect this isolated term (they only relabel outputs). We will add subsequent interfaces in later steps.\\

\noindent\textbf{Step 2 (layers \(2\) to \(L\!-\!1\): layer-wise Frobenius costs; fix \(S_{0}=I\)).}
For each interior polynomial layer \(\ell=2,\dots,L-1\), let \(D_{\ell-1}=\mathrm{diag}(d_{\ell-1,1},\dots,d_{\ell-1,d_{\ell-1}})\), \(D_{\ell}=\mathrm{diag}(d_{\ell,1},\dots,d_{\ell,d_{\ell}})\), \(s_{\ell}:=D_{\ell}^{1/\alpha_{\ell}}\), and \(\Pi_{\ell}\in\mathcal{S}_{d_{\ell}}\). Under the invariance,
- inputs of \(P_{\ell}\) are scaled by \(D_{\ell-1}^{-1}\) (permutation does not affect Euclidean norms),
- outputs of \(P_{\ell}\) are permuted by \(\Pi_{\ell}\) and scaled componentwise by \(s_{\ell}^{\,1/r^{(\ell)}}\).

Define, for each output and term,
\[
  a^{(\ell)}_{j,k}(D_{\ell-1})\;:=\;\big\|D_{\ell-1}^{-1} w^{(\ell)}_{j,k}\big\|_{2}^{2}\;+\;\mu\,\big(b^{(\ell)}_{j,k}\big)^{2},~
  \gamma^{(\ell)}_{j,k}\;:=\;\frac{2}{\alpha_{\ell}\, r^{(\ell)}_{j,k}}.
\]
Then the Frobenius contribution of layer \(\ell\) after \((D_{\ell-1},\Pi_{\ell},D_{\ell})\) is
\[
  \mathcal{R}^{(\ell)}_{\mathrm{F}}(D_{\ell-1},\Pi_{\ell},D_{\ell})
  \;=\;
  \sum_{i=1}^{d_{\ell}}\;\sum_{k=1}^{K_{\ell,\Pi_{\ell}(i)}}
  a^{(\ell)}_{\Pi_{\ell}(i),k}(D_{\ell-1})\;\; d_{\ell,i}^{\,\gamma^{(\ell)}_{\Pi_{\ell}(i),k}}.
\]
Summing over \(\ell=2,\dots,L-1\) gives the step‑2 cost
\[
  \mathcal{R}^{(2{:}L\!-\!1)}_{\mathrm{F}}\big(\{D_{\ell-1},\Pi_{\ell},D_{\ell}\}_{\ell=2}^{L-1}\big)
  \;=\;\sum_{\ell=2}^{L-1}\mathcal{R}^{(\ell)}_{\mathrm{F}}(D_{\ell-1},\Pi_{\ell},D_{\ell}).
\]
We now move on to the determining the cost for the last layer.\\

\noindent\textbf{Step 3 (last polynomial layer \(P_{L}\): Frobenius cost).}
At the final interface, only the precomposition by \(B_{L-1}=(\Pi_{L-1}D_{L-1})^{-1}\) acts on \(P_{L}\); there is no post-scaling. Since \(\Pi_{L-1}\) is orthogonal, the Euclidean norm depends only on \(D_{L-1}\). Define
\[
  C^{(L-1)}_{i}\;:=\;\sum_{j=1}^{d_{L}}\sum_{k=1}^{K_{L,j}}\big(w^{(L)}_{j,k}[i]\big)^{2},\qquad i=1,\dots,d_{L-1}.
\]
Then
\[
\begin{aligned}
  \mathcal{R}^{(L)}_{\mathrm{F}}(D_{L-1})
  &= \sum_{j=1}^{d_{L}}\sum_{k=1}^{K_{L,j}}\Big(\,\|D_{L-1}^{-1} w^{(L)}_{j,k}\|_{2}^{2}+\mu\,(b^{(L)}_{j,k})^{2}\,\Big)\\
  &= \sum_{i=1}^{d_{L-1}} C^{(L-1)}_{i}\, d_{L-1,i}^{-2}\;+\;\mu\sum_{j,k}(b^{(L)}_{j,k})^{2},
\end{aligned}
\]
where the bias term is invariant to \(D_{L-1}\) and can be treated as a constant for optimization over \(D_{L-1}\). No \(D_{L}\) or \(\Pi_{L-1}\) dependence remains in this layer’s cost.\\

\noindent\textbf{Step 4 (Geometric program for regularizer minimization; fix \(\{\Pi_{\ell}\}\)).}
Decision variables: positive diagonals \(D_{\ell}=\mathrm{diag}(d_{\ell,1},\dots,d_{\ell,d_{\ell}})\) for \(\ell=1,\dots,L-1\). With fixed permutations \(\{\Pi_{\ell}\}\), the total Frobenius regularizer
\[
\begin{aligned}
\mathcal{R}_{\mathrm{F}}^{\mathrm{tot}}(\{D_{\ell}\})
&= \underbrace{\mathcal{R}^{(1)}_{\mathrm{F}}(D_{1})}_{\text{Step 1}}
+ \underbrace{\sum_{\ell=2}^{L-1}\mathcal{R}^{(\ell)}_{\mathrm{F}}(D_{\ell-1},\Pi_{\ell},D_{\ell})}_{\text{Step 2}}\\
&\quad+ \underbrace{\mathcal{R}^{(L)}_{\mathrm{F}}(D_{L-1})}_{\text{Step 3}}
\end{aligned}
\]
is a posynomial in \(\{d_{\ell,i}\}\). Concretely, using the per-layer components:\\
- First layer (Step 1):
\[
\begin{aligned}
\mathcal{R}^{(1)}_{\mathrm{F}}(D_{1})
&= \sum_{j=1}^{d_{1}}\sum_{k=1}^{K_{1,j}} a^{(1)}_{j,k}\, d_{1,j}^{\,\gamma^{(1)}_{j,k}},\\
a^{(1)}_{j,k}&:= \|w^{(1)}_{j,k}\|_{2}^{2}+\mu(b^{(1)}_{j,k})^{2},\\
\gamma^{(1)}_{j,k}&:= \tfrac{2}{\alpha_{1}\, r^{(1)}_{j,k}}.
\end{aligned}
\]
- Interior layers (Step 2), writing \(j=\Pi_{\ell}(i)\), $\mathcal{R}^{(\ell)}_{\mathrm{F}}(D_{\ell-1},\Pi_{\ell},D_{\ell})$ can be written as
\[
\begin{aligned}
  & \sum_{i=1}^{d_{\ell}}\sum_{k=1}^{K_{\ell,j}}
     \Big(\underbrace{\sum_{p=1}^{d_{\ell-1}} (w^{(\ell)}_{j,k}[p])^{2}\, d_{\ell-1,p}^{-2}
     \;+\;\mu(b^{(\ell)}_{j,k})^{2}d_{\ell,i}^{\,\gamma^{(\ell)}_{j,k}}}_{\text{posynomial in }D_{\ell-1}^{-1}}\Big)\;
      \\
  &\text{with }\ \gamma^{(\ell)}_{j,k}:=\tfrac{2}{\alpha_{\ell} r^{(\ell)}_{j,k}},\quad
  j:=\Pi_{\ell}(i),\quad \ell=2,\dots,L-1.
\end{aligned}
\]
- Last layer (Step 3):
\[
\begin{aligned}
  \mathcal{R}^{(L)}_{\mathrm{F}}(D_{L-1})
  &= \sum_{i=1}^{d_{L-1}} C^{(L-1)}_{i}\, d_{L-1,i}^{-2}
   + \underbrace{\mu\sum_{j,k}(b^{(L)}_{j,k})^{2}}_{\text{constant}},\\
C^{(L-1)}_{i}&:= \sum_{j,k}(w^{(L)}_{j,k}[i])^{2}.
\end{aligned}
\]
Hence, the invariance-aware regularizer minimization is the geometric programm given by:
\[
\begin{aligned}
  \min_{\{d_{\ell,i}>0\}} \quad & \mathcal{R}_{\mathrm{F}}^{\mathrm{tot}}(\{D_{\ell}\}) \\
  \text{s.t.}\quad & \prod_{i=1}^{d_{\ell}} d_{\ell,i}=1,\quad \ell=1,\dots,L-1\quad\text{(optional anchors)},\\
                   & d_{\min}\ \le\ d_{\ell,i}\ \le\ d_{\max}\quad\text{(optional bounds)}.
\end{aligned}
\]
Anchors are monomial equalities and bounds are posynomial/monomial constraints; both preserve GP structure. After the standard log change \(t_{\ell,i}=\log d_{\ell,i}\), the problem becomes convex and can be solved by any standard convex program solver. Discrete permutations \(\{\Pi_{\ell}\}\) can be fixed a priori or updated by an outer assignment step (cf. Step 2); for fixed \(\{\Pi_{\ell}\}\) the continuous subproblem above is a convex GP.\\

With symmetric anchors/bounds (coordinate-wise identical), the optimal value of the GP is independent of the permutations:
\[
\begin{aligned}
\min_{\{d_{\ell,i}>0\}}~ \mathcal{R}_{\mathrm{F}}^{\mathrm{tot}}\!\big(\{D_{\ell}\},\{\Pi_{\ell}\}\big)
&= \min_{\{d_{\ell,i}>0\}}~ \mathcal{R}_{\mathrm{F}}^{\mathrm{tot}}\!\big(\{D_{\ell}\},\{I\}\big).
\end{aligned}
\]
Note again, that at each layer, changing \(\Pi_{\ell}\) only reindexes which diagonal entry \(d_{\ell,i}\) multiplies a given term; since \(\{d_{\ell,i}\}\) are free variables with symmetric constraints, a corresponding relabeling of \(\{d_{\ell,i}\}\) attains the same objective value.\\

\noindent\emph{L1 regularization}: Replacing Frobenius terms by 1-norms yields the same GP structure. Define \(\eta^{(\ell)}_{j,k}:=\tfrac{1}{\alpha_{\ell}\,r^{(\ell)}_{j,k}}\).
- First layer:
\[
\begin{aligned}
\mathcal{R}^{(1)}_{1}(D_{1})
&= \sum_{j=1}^{d_{1}}\sum_{k=1}^{K_{1,j}}
   \big(\|w^{(1)}_{j,k}\|_{1}+\mu\,|b^{(1)}_{j,k}|\big)\, d_{1,j}^{\,\eta^{(1)}_{j,k}}.
\end{aligned}
\]
- Interior layers (\(\ell=2,\dots,L-1\)), with \(j:=\Pi_{\ell}(i)\):
\[
\begin{aligned}
\mathcal{R}^{(\ell)}_{1}(D_{\ell-1},\Pi_{\ell},D_{\ell})
&= \sum_{i=1}^{d_{\ell}}\sum_{k=1}^{K_{\ell,j}}
   \Big(\sum_{p=1}^{d_{\ell-1}} |w^{(\ell)}_{j,k}[p]|\, d_{\ell-1,p}^{-1}
        + \mu\,|b^{(\ell)}_{j,k}|\Big)\, d_{\ell,i}^{\,\eta^{(\ell)}_{j,k}}.
\end{aligned}
\]
- Last layer:
\[
\begin{aligned}
\mathcal{R}^{(L)}_{1}(D_{L-1})
&= \sum_{i=1}^{d_{L-1}} C^{(L-1)}_{1,i}\, d_{L-1,i}^{-1}
   + \mu\sum_{j,k}|b^{(L)}_{j,k}|,\\
C^{(L-1)}_{1,i}&:= \sum_{j=1}^{d_{L}}\sum_{k=1}^{K_{L,j}} |w^{(L)}_{j,k}[i]|.
\end{aligned}
\]
Each term is a posynomial in \(\{d_{\ell,i}\}\); the constraints are monomial/posynomial. After the log change, the L1-regularized problem remains a convex GP. The permutation-independence argument applies verbatim.

\section{Parameter Obfuscation}

\noindent Consider two parties, Alice and Bob. Alice holds an input \(x\in\mathbb{R}^{d_{0}}\) and Bob holds the model parameters \(\theta\) defining \(f_{\theta}\). The goal is to compute \(f_{\theta}(x)\) while ensuring that Bob does not learn \(x\) and Alice does not learn \(\theta\).\\

\noindent A simple way for Alice to protect her input is to apply a random invertible matrix \(R\in \mathrm{GL}(d_{0})\) known only to her and reveal only the obfuscated vector
\[
  \tilde x \;=\; R\,x,
\]
and to any external party (including Bob). Access is thus restricted to \(R x\), never to \(x\) itself. The choice of \(R\) can be, e.g., a random orthogonal or well‑conditioned Gaussian matrix to preserve numerical stability. Further protocol steps (for Bob to act on \(\tilde x\) without revealing \(\theta\)) will build on this transform.

\noindent On Bob’s side, only a session‑specific representative from the invariance class of \(\theta\) needs to be shared. Using the layerwise equivalence (input reparameterization and inter‑layer \(\Pi_{\ell},D_{\ell}\)), Bob can publish parameters \(\hat\theta\) that implement exactly the same map as \(f_{\theta}\), while absorbing Alice’s obfuscation \(R\) into the first layer via \(P_{1}\!\mapsto\! P_{1}\circ R^{-1}\). Evaluation on \(\tilde x=R x\) then reproduces \(f_{\theta}(x)\) without disclosing the original \(\theta\).\\

\noindent To prevent Alice from caching a reusable parameter copy, Bob can add a secret, per‑session polarity mask at each hidden interface. Define a diagonal mask \(S_{\ell}=\mathrm{diag}(s_{\ell,1},\dots,s_{\ell,d_{\ell}})\) with \(s_{\ell,i}\in\{\pm 1\}\), and use the identity
\[
  (x)_{+}^{\alpha}\;=\;(-x)_{-}^{\alpha},\qquad (x)_{-}:=(-x)_{+}.
\]
Equivalently, write a polarity‑modulated activation
\[
  [M_{\ell}^{(S_{\ell})}(z)]_{i}\;=\;(s_{\ell,i}\,z_{i})_{+}^{\alpha_{\ell}},
\]
and absorb \(S_{\ell}\) into adjacent polynomials using the same push‑through rules as for \(D_{\ell}\): insert \(S_{\ell}^{\,1/\alpha_{\ell}}\) after \(P_{\ell}\) and precompose \(P_{\ell+1}\) by \((\Pi_{\ell}\,S_{\ell}\,D_{\ell})^{-1}\). This keeps the realized map unchanged while the session‑specific signs \(\{S_{\ell}\}\) remain hidden; a parameter set cached without the current polarity masks will not be valid for future sessions.\\

\noindent The orchestration proceeds as follows.
1) Bob samples per‑layer transformations \(\{\Pi^{\mathrm{B}}_{\ell},D^{\mathrm{B}}_{\ell},S^{\mathrm{B}}_{\ell}\}_{\ell=1}^{L-1}\) and sets \(S^{\mathrm{B}}_{0}=I\). He forms a session‑specific equivalent parameter set \(\hat\theta\) by applying the invariance map generated by these choices to \(\theta\), then sends \(\hat\theta\) to Alice.
2) Alice samples a random \(R\in\mathrm{GL}(d_{0})\) and sets \(S^{\mathrm{A}}_{0}:=R\). She also samples her own \(\{\Pi^{\mathrm{A}}_{\ell},D^{\mathrm{A}}_{\ell},S^{\mathrm{A}}_{\ell}\}_{\ell=1}^{L-1}\) and applies the invariance map to obtain \(\tilde\theta\) from \(\hat\theta\). She returns \(\tilde\theta\) to Bob and, for each query, sends the obfuscated input \(\tilde x:=R x\).
3) Bob evaluates \(y=f_{\tilde\theta}(\tilde x)\) and returns \(y\) to Alice. By equivariance and the composition of invariances, \(f_{\tilde\theta}(R x)=f_{\theta}(x)\) for all \(x\).\\

\noindent Why Bob cannot recover \(R\). From Bob’s perspective, the pair \((\hat\theta,\tilde\theta)\) reveals only that there exists some invariance transform composed of Alice’s unknown \((R,\{\Pi^{\mathrm{A}}_{\ell},D^{\mathrm{A}}_{\ell},S^{\mathrm{A}}_{\ell}\})\) such that \(\tilde\theta=T_{\mathrm{A}}(\hat\theta)\). The factorization of \(T_{\mathrm{A}}\) into the first‑layer input map \(R\) and the inter‑layer permutations/diagonals/polarity masks is non‑identifiable: many tuples \((R',\{\Pi'_{\ell},D'_{\ell},S'_{\ell}\})\) produce the same \(\tilde\theta\). In particular, right‑multiplying \(R\) by any permutation/positive diagonal (and compensating in adjacent layers) or flipping hidden signs yields the same published \(\tilde\theta\). Thus \(R\) is only identifiable up to a large invariance group known to Alice but hidden from Bob. Resampling \(R\) per session further prevents linkage across sessions; Bob only ever sees \(\tilde x=R x\) and \(\tilde\theta\), never \(x\) or \(\theta\).

\section{Parameter Range Minimization}
\noindent The goal is to select layerwise positive diagonal reparameterizations (the same family used in the invariance group) so that, after reparameterization, the numeric range of the parameters at each layer is as small as possible. Minimizing parameter ranges improves quantization behavior and reduces the dynamic range required for representation and communication.

\subsection{Setup and Notation}
Fix an interior interface $\ell\in\{1,\dots,L-1\}$. As before, let
\[
D_{\ell}=\mathrm{diag}(d_{\ell,1},\dots,d_{\ell,d_{\ell}}),\qquad d_{\ell,i}>0,
\]
and define
\[
s_{\ell,j}:=d_{\ell,j}^{\,1/\alpha_{\ell}},\qquad j=1,\dots,d_{\ell}.
\]
For layer $\ell$, consider an output $j$ and one of its polynomial terms indexed by $k$ with degree $r^{(\ell)}_{j,k}$. The affine parameters of that term are
\[
(w^{(\ell)}_{j,k}[p])_{p=1}^{d_{\ell-1}} \quad \text{(weights)}, \qquad b^{(\ell)}_{j,k}\quad \text{(bias)}.
\]
Under the invariance reparameterization that (i) post-scales outputs of $P_{\ell}$ by $\Pi_{\ell}D_{\ell}^{1/\alpha_{\ell}}$ and (ii) precomposes $P_{\ell+1}$ by $(\Pi_{\ell}D_{\ell})^{-1}$, each scalar parameter is multiplied by a positive scalar. After fixing permutations $\{\Pi_{\ell}\}$ and using only diagonals, the $(p)$-th entry of the transformed weight is
\[
\widetilde w^{(\ell)}_{j,k}[p]
= w^{(\ell)}_{j,k}[p]\; d_{\ell-1,p}^{-1}\; d_{\ell,j}^{\,\gamma^{(\ell)}_{j,k}},
\qquad
\gamma^{(\ell)}_{j,k} := \frac{1}{\alpha_{\ell}\, r^{(\ell)}_{j,k}},
\]
while the transformed bias is
\[
\widetilde b^{(\ell)}_{j,k}
= b^{(\ell)}_{j,k}\; d_{\ell,j}^{\,\gamma^{(\ell)}_{j,k}}.
\]
Signs are preserved because the diagonals are strictly positive.

\subsection{Positive/Negative Partition and Layerwise Bounds}
For each layer $\ell$, collect all scalar parameters (entries of all $w^{(\ell)}_{j,k}$ and $b^{(\ell)}_{j,k}$) and partition them by sign:
\[
\mathcal{P}^{(\ell)} := \{a : a \ge 0\}, \qquad
\mathcal{N}^{(\ell)} := \{a : a < 0\}.
\]
Zero entries may be placed in $\mathcal{P}^{(\ell)}$ or handled separately using a small tolerance. After applying the diagonal scalings $\{D_{\ell-1},D_{\ell}\}$, transformed positive parameters remain nonnegative and transformed negative parameters remain nonpositive. Introduce per-layer nonnegative bounds
\begin{equation} \nonumber
s^{(\ell)}_{p}\ge 0 \quad \text{(upper bound on transformed positives)}, \end{equation}
and
\begin{equation} \nonumber
s^{(\ell)}_{n}\ge 0 \quad \text{(upper bound on minus transformed negatives)},
\end{equation}
defined by
\[
\begin{aligned}
&\forall\, a\in\mathcal{P}^{(\ell)}:\quad a_{\text{transformed}} \le s^{(\ell)}_{p},\\
&\forall\, a\in\mathcal{N}^{(\ell)}:\quad -a_{\text{transformed}} \le s^{(\ell)}_{n}.
\end{aligned}
\]
Hence $s^{(\ell)}_{p}+s^{(\ell)}_{n}$ upper bounds the total numerical span of parameters in layer $\ell$.

\subsection{Optimization Problem}
We now seek diagonal scalings $\{D_{\ell}\}_{\ell=1}^{L-1}$ and bounds $\{s^{(\ell)}_{p},s^{(\ell)}_{n}\}$ that minimize the worst per-layer range. Introducing a global scalar $t>0$, we pose
\[
\begin{aligned}
\min_{\substack{\{d_{\ell,i}>0\}\\ \{s^{(\ell)}_{p},s^{(\ell)}_{n}\ge 0\}\\ t>0}} \quad & t \\[3pt]
\text{s.t.}\quad
& s^{(\ell)}_{p}+s^{(\ell)}_{n} \le t,\qquad \ell=1,\dots,L-1,\\[3pt]
& s^{(\ell)}_{p} \ge a \cdot \Phi_{a}(\{d_{\ell-1,p}\},d_{\ell,j}), \quad a\in\mathcal{P}^{(\ell)},\\
& s^{(\ell)}_{n} \ge (-a) \cdot \Phi_{a}(\{d_{\ell-1,p}\},d_{\ell,j}), \quad a\in\mathcal{N}^{(\ell)},\\[3pt]
& \prod_{i=1}^{d_{\ell}} d_{\ell,i}=1,\qquad
d_{\min} \le d_{\ell,i} \le d_{\max}\quad \text{(optional)}.
\end{aligned}
\]
Here $\Phi_{a}(\cdot)$ is the multiplicative factor by which the scalar parameter $a$ is scaled under the invariance transformation. For a weight entry $w^{(\ell)}_{j,k}[p]$,
\[
\Phi_{w^{(\ell)}_{j,k}[p]}(\{d_{\ell-1}\},d_{\ell,j})
= d_{\ell-1,p}^{-1}\; d_{\ell,j}^{\,\gamma^{(\ell)}_{j,k}},
\]
and for a bias $b^{(\ell)}_{j,k}$,
\[
\Phi_{b^{(\ell)}_{j,k}}(d_{\ell,j})
= d_{\ell,j}^{\,\gamma^{(\ell)}_{j,k}}.
\]

\subsection{Convexity and Geometric-Program Structure}
All inequalities above are monomial or posynomial relations between positive variables. After rearrangement, each takes the form of a posynomial upper-bound inequality, hence the overall problem is a \emph{geometric program} (GP). With the logarithmic change of variables
\[
u_{\ell,i}=\log d_{\ell,i},\quad
\tilde s^{(\ell)}_{p}=\log s^{(\ell)}_{p},\quad
\tilde s^{(\ell)}_{n}=\log s^{(\ell)}_{n},\quad
\tilde t=\log t,
\]
the GP becomes a convex optimization problem solvable by standard GP solvers. Optional anchors $\prod_i d_{\ell,i}=1$ and bounds on $d_{\ell,i}$ preserve convexity.

\subsection{Remarks}
\begin{itemize}
  \item \textbf{Permutations.} As before, permutations $\{\Pi_{\ell}\}$ only relabel which diagonal entries multiply given coefficients. If $\{\Pi_{\ell}\}$ are not fixed, one may alternate between (i) solving the GP for fixed $\{\Pi_{\ell}\}$ and (ii) reassigning via an outer permutation update.
  \item \textbf{Zeros and Stability.} Zero coefficients lead to trivial constraints. To ensure numerical stability, maintain $d_{\min}>0$ and optionally clip very large/small $d_{\ell,i}$.
  \item \textbf{Polarity Masks.} If session-specific sign masks $S_{\ell}$ are used (as in the obfuscation protocol), the partition into $\mathcal{P}^{(\ell)}$ and $\mathcal{N}^{(\ell)}$ should reflect the masked parameters. The GP structure is unaffected.
  \item \textbf{Objective Interpretation.} The objective $t$ minimizes the maximum (across layers) of the per-layer sum $s^{(\ell)}_{p}+s^{(\ell)}_{n}$, i.e., the largest parameter span. Weighted or additive variants can be handled analogously.
  \item \textbf{Implementation.} The GP can be implemented directly in CVXPY, CVX, or GPkit. During gradient-based training, the reparameterization can be recomputed periodically (every few iterations) and cached for efficiency.
\end{itemize}

\noindent This formulation provides a principled, invariance-aware method to determine diagonal scalings that minimize layerwise parameter ranges, thereby improving numerical robustness and compactness without altering the realized input–output map.

\section{Obfuscating Remote Training}
\noindent Consider the scenario where a party holds a private dataset and wishes to train a neural network on a remote computational resource (e.g., a cloud server) without revealing either the training data or the final model parameters to the remote party. Using the invariance structure developed in earlier sections, a protocol is constructed that achieves both data and model obfuscation for multilayer perceptrons (MLPs) in classification and regression settings.

\subsection{Problem Setup}
Let the data holder possess a training set $\{(x^{(n)}, y^{(n)})\}_{n=1}^{N}$ with inputs $x^{(n)}\in\mathbb{R}^{d_{0}}$ and labels $y^{(n)}$ (class indices for classification, or vectors for regression). The goal is to train an $L$-layer MLP
\[
  f_{\theta}(x) \;=\; W_{L}\,\sigma_{L-1}(W_{L-1}\cdots\sigma_{1}(W_{1}x + b_{1})\cdots + b_{L-1}) + b_{L},
\]
where $W_{\ell}\in\mathbb{R}^{d_{\ell}\times d_{\ell-1}}$, $b_{\ell}\in\mathbb{R}^{d_{\ell}}$, and $\sigma_{\ell}$ denotes a componentwise activation (e.g., ReLU). The remote server will perform the training iterations, but neither the original data nor the true model parameters should be recoverable by the server.

\subsection{Input Obfuscation}
The data holder samples a secret invertible matrix $R\in\mathrm{GL}(d_{0})$ and transforms all training inputs:
\[
  \tilde x^{(n)} \;:=\; R\, x^{(n)}, \qquad n=1,\dots,N.
\]
The matrix $R$ may be chosen as a random orthogonal matrix (to preserve norms) or a well-conditioned Gaussian matrix. Only the transformed dataset $\{(\tilde x^{(n)}, \tilde y^{(n)})\}$ is transmitted to the server; the original inputs $\{x^{(n)}\}$ and the matrix $R$ are never disclosed.

\subsection{First-Layer Compensation}
To ensure that training on the transformed inputs yields equivalent gradients and loss values, the first-layer weight matrix is pre-multiplied by $R^{-1}$. Concretely, initialize $W_{1}$ with random values and define
\[
  \tilde W_{1} \;:=\; W_{1}\, R^{-1}.
\]
Then for any input,
\[
  \tilde W_{1}\,\tilde x \;=\; W_{1}\, R^{-1}\, R\, x \;=\; W_{1}\, x,
\]
so the pre-activation at the first hidden layer is identical whether one uses $(W_{1}, x)$ or $(\tilde W_{1}, \tilde x)$. The data holder sends $\tilde W_{1}$ (not $W_{1}$) to the server.

\subsection{Hidden-Layer Obfuscation via Permutation and Diagonal Scaling}
For each hidden interface $\ell=1,\dots,L-1$, the data holder samples:
\begin{itemize}
  \item a random permutation matrix $\Pi_{\ell}\in\mathcal{S}_{d_{\ell}}$,
  \item a random positive diagonal matrix $D_{\ell}=\mathrm{diag}(d_{\ell,1},\dots,d_{\ell,d_{\ell}})$ with $d_{\ell,i}>0$.
\end{itemize}
Using the equivariance identity for ReLU-type activations,
\[
  \sigma(\Pi_{\ell}\,D_{\ell}\,z) \;=\; \Pi_{\ell}\,D_{\ell}\,\sigma(z),
\]
(since permutations and positive scalings commute with componentwise nonnegative activations), the weight matrices are transformed as follows. Define $T_{\ell}:=\Pi_{\ell}D_{\ell}$ and set
\[
\begin{aligned}
  \tilde W_{1} &:= T_{1}\, W_{1}\, R^{-1}, \\
  \tilde W_{\ell} &:= T_{\ell}\, W_{\ell}\, T_{\ell-1}^{-1}, \quad \ell=2,\dots,L-1, \\
  \tilde W_{L} &:= W_{L}\, T_{L-1}^{-1}.
\end{aligned}
\]
Biases are similarly transformed: $\tilde b_{\ell}:=T_{\ell}\,b_{\ell}$ for $\ell=1,\dots,L-1$ and $\tilde b_{L}:=b_{L}$. The server receives the obfuscated initialization $\{\tilde W_{\ell}, \tilde b_{\ell}\}$; the secret transformations $\{R, \Pi_{\ell}, D_{\ell}\}$ remain with the data holder.

\subsection{Label Permutation (Classification)}
For a $C$-class classification problem, the data holder samples a random class permutation $\pi\in\mathcal{S}_{C}$ and relabels all training targets:
\[
  \tilde y^{(n)} \;:=\; \pi(y^{(n)}).
\]
Correspondingly, the rows of the final weight matrix $W_{L}$ and bias $b_{L}$ are permuted by $\pi$ before transmission, so that the network output is consistently permuted. After training, the inverse permutation $\pi^{-1}$ recovers the true class ordering.

\subsection{Output Transformation (Regression / MSE Loss)}
For regression with mean-squared-error loss, the data holder may sample a random orthogonal matrix $Q\in\mathrm{O}(d_{L})$ and transform both the network output and the target:
\[
  \tilde y^{(n)} \;:=\; Q\, y^{(n)}, \qquad \tilde W_{L} \;:=\; Q\, W_{L}\, T_{L-1}^{-1}, \qquad \tilde b_{L} \;:=\; Q\, b_{L}.
\]
Since orthogonal transformations preserve Euclidean distances, the MSE loss is invariant:
\[
  \|Q\,(\hat y - y)\|_{2}^{2} \;=\; \|\hat y - y\|_{2}^{2}.
\]
Gradients computed on the server thus have identical norms and directions (up to the rotation $Q$), ensuring that training dynamics are unaffected.

\subsection{Remote Training and Parameter Recovery}
The server receives the obfuscated dataset $\{\tilde x^{(n)}, \tilde y^{(n)}\}$ and the obfuscated initial parameters $\{\tilde W_{\ell}, \tilde b_{\ell}\}$. Standard gradient-based training (SGD, Adam, etc.) is performed on the server, yielding trained parameters $\{\tilde W_{\ell}^{*}, \tilde b_{\ell}^{*}\}$.

Upon return, the data holder recovers the true parameters by inverting the transformations:
\[
\begin{aligned}
  W_{1}^{*} &= T_{1}^{-1}\,\tilde W_{1}^{*}\, R, \\
  W_{\ell}^{*} &= T_{\ell}^{-1}\,\tilde W_{\ell}^{*}\, T_{\ell-1}, \quad \ell=2,\dots,L-1, \\
  W_{L}^{*} &= Q^{-1}\,\tilde W_{L}^{*}\, T_{L-1} \quad \text{(regression)}, \\
  W_{L}^{*} &= \Pi_{C}^{-1}\,\tilde W_{L}^{*}\, T_{L-1} \quad \text{(classification)},
\end{aligned}
\]
and similarly for biases. Here $\Pi_{C}$ denotes the permutation matrix corresponding to $\pi$. The recovered network $f_{\theta^{*}}$ operates on unobfuscated inputs $x$ and produces correct outputs $y$.

\subsection{Security Properties}
\begin{itemize}
  \item \textbf{Data obfuscation.} The server only observes $\{\tilde x^{(n)}\}=\{R\,x^{(n)}\}$. Without knowledge of $R$, the original inputs cannot be recovered; $R$ acts as a one-time pad in the input space.
  \item \textbf{Label obfuscation.} Class labels are permuted by $\pi$ (classification) or rotated by $Q$ (regression). The server cannot determine the true label semantics.
  \item \textbf{Model obfuscation.} The server trains and returns parameters in the obfuscated representation. Without $\{R, \Pi_{\ell}, D_{\ell}, Q\}$, the returned $\{\tilde W_{\ell}^{*}\}$ cannot be converted to a usable model; any attempt to use them directly on original inputs produces meaningless outputs.
  \item \textbf{Non-transferability.} Even if the server caches the trained parameters, they are bound to the secret transformations held by the data owner. A different owner (or the same owner with refreshed secrets) cannot reuse cached parameters.
\end{itemize}

\subsection{Practical Considerations}
\begin{itemize}
  \item \textbf{Communication.} The protocol adds no communication overhead beyond the initial parameter exchange; gradient updates and loss values are computed entirely on the server.
  \item \textbf{Extensions.} The approach extends to convolutional layers (via channel permutations and per-channel scaling), attention mechanisms, and other architectures that exhibit analogous equivariance properties.
\end{itemize}

\noindent This protocol leverages the invariance group structure to enable obfuscated-outsourced training: the remote server performs standard optimization on transformed data and parameters, while the data holder retains exclusive knowledge of the secret transformations required to interpret the trained model.

\section{Parametric Invariances in Self-Attention}

In this section, we describe two fundamental parameter invariances of the self-attention mechanism: a bilinear invariance associated with the query--key interaction, and a linear invariance associated with the value--output projection. These invariances hold independently of residual connections, normalization layers, and feed-forward sublayers, and reflect intrinsic non-identifiability in the attention parametrization.

\subsection{Query--Key Bilinear Invariance}

Consider a self-attention layer with projections
\begin{equation}
Q = X W_Q, \qquad K = X W_K,
\end{equation}
where \(X \in \mathbb{R}^{n \times d}\), \(W_Q, W_K \in \mathbb{R}^{d \times d_k}\).
The attention scores are given by
\begin{equation}
S = Q K^\top = X W_Q W_K^\top X^\top,
\end{equation}
and the attention weights are
\begin{equation}
A = \mathrm{softmax}\!\left(\frac{S}{\sqrt{d_k}}\right).
\end{equation}

\begin{lemma}[Query--Key Bilinear Invariance]
Let \(P \in \mathbb{R}^{d_k \times d_k}\) be any invertible matrix. Define transformed parameters
\begin{equation}
W_Q' = W_Q P, \qquad W_K' = W_K P^{-T}.
\end{equation}
Then the attention scores and attention weights are unchanged, i.e.,
\begin{equation}
Q' K'^\top = Q K^\top, \qquad A' = A.
\end{equation}
\end{lemma}

\begin{proof}
Under the parameter transformation,
\begin{align}
Q' &= X W_Q P, \\
K' &= X W_K P^{-T}.
\end{align}
Hence,
\begin{equation}
Q' K'^\top
= X W_Q P (P^{-1} W_K^\top) X^\top
= X W_Q W_K^\top X^\top
= Q K^\top.
\end{equation}
Since the softmax is applied elementwise to the rows of the score matrix, identical scores yield identical attention weights.
\end{proof}

This invariance reflects the fact that self-attention depends only on the bilinear form induced by \(W_Q W_K^\top\), and not on the individual bases used to represent queries and keys.

---

\subsection{Value--Output Linear Invariance}

Let the value projection and output projection be defined as
\begin{equation}
V = X W_V, \qquad Y = A V W_O,
\end{equation}
where \(W_V \in \mathbb{R}^{d \times d_v}\) and \(W_O \in \mathbb{R}^{d_v \times d}\).

\begin{lemma}[Value--Output Linear Invariance]
Let \(R \in \mathbb{R}^{d_v \times d_v}\) be any invertible matrix. Define transformed parameters
\begin{equation}
W_V' = W_V R, \qquad W_O' = R^{-1} W_O.
\end{equation}
Then the attention output is unchanged, i.e.,
\begin{equation}
A V' W_O' = A V W_O.
\end{equation}
\end{lemma}

\begin{proof}
Under the transformation,
\begin{equation}
V' = X W_V' = X W_V R = V R.
\end{equation}
The attention aggregation yields
\begin{equation}
A V' = A (V R) = (A V) R.
\end{equation}
Applying the transformed output projection,
\begin{equation}
A V' W_O' = (A V) R (R^{-1} W_O) = A V W_O,
\end{equation}
which proves the claim.
\end{proof}

This invariance shows that the value space admits an arbitrary change of basis, provided it is undone by the output projection. Consequently, only the composition \(W_V W_O\) is functionally identifiable from the attention output.

---

\subsection{Discussion}

The query--key bilinear invariance and the value--output linear invariance together imply that the self-attention mechanism possesses a large continuous symmetry group. These symmetries are independent of residual connections and normalization, and they explain why individual attention heads, value directions, or projection matrices are not uniquely identifiable. Subsequent architectural components such as LayerNorm, nonlinear feed-forward networks, and the final language-model head partially or completely break these invariances.

\subsection{Permutation Equivariance of a Pre-LN Transformer Block}

We consider a single Transformer block with the \emph{pre-layer-normalization (Pre-LN)} architecture, defined by
\begin{align}
Y &= X + \mathrm{SA}(\mathrm{LN}(X)), \label{eq:preln_sa}\\
Z &= Y + \mathrm{FFN}(\mathrm{LN}(Y)), \label{eq:preln_ffn}
\end{align}
where \(X \in \mathbb{R}^{n \times d}\) is the input sequence representation, \(\mathrm{SA}\) denotes multi-head self-attention including the output projection, \(\mathrm{FFN}\) is a position-wise feed-forward network with an elementwise nonlinearity, and \(\mathrm{LN}\) denotes Layer Normalization applied independently to each token.

\begin{theorem}[Permutation Equivariance]
Let \(P \in \mathbb{R}^{d \times d}\) be a permutation matrix acting on the feature dimension, and define the transformed input \(X' = X P\). Suppose the model parameters are transformed as
\begin{align}
W_Q' &= P^{-1} W_Q, &
W_K' &= P^{-1} W_K, \nonumber\\
W_V' &= P^{-1} W_V P, &
W_O' &= P^{-1} W_O P, \label{eq:attn_param_transform}\\
W_1' &= P^{-1} W_1, &
W_2' &= W_2 P. \label{eq:ffn_param_transform}
\end{align}
Then the Transformer block satisfies the equivariance relation
\begin{equation}
\mathrm{Block}(X P) = \mathrm{Block}(X)\, P.
\end{equation}
\end{theorem}

\begin{proof}
We proceed component by component.

\paragraph{LayerNorm equivariance.}
For a vector \(x \in \mathbb{R}^d\), LayerNorm is defined as
\begin{equation}
\mathrm{LN}(x) = \gamma \odot \frac{x - \mu(x)}{\sigma(x)} + \beta,
\end{equation}
where
\(
\mu(x) = \frac{1}{d}\sum_{i=1}^d x_i
\)
and
\(
\sigma(x) = \sqrt{\frac{1}{d}\sum_{i=1}^d (x_i - \mu(x))^2}.
\)
Since permutations preserve both the mean and variance, we have
\begin{equation}
\mu(xP) = \mu(x), \qquad \sigma(xP) = \sigma(x),
\end{equation}
which implies
\begin{equation}
\mathrm{LN}(xP) = \mathrm{LN}(x)\,P.
\end{equation}
Applying this row-wise yields
\begin{equation}
\mathrm{LN}(X P) = \mathrm{LN}(X)\,P.
\end{equation}

\paragraph{Self-attention equivariance.}
Let
\begin{equation}
Q = \mathrm{LN}(X) W_Q, \quad
K = \mathrm{LN}(X) W_K, \quad
V = \mathrm{LN}(X) W_V,
\end{equation}
and define the attention weights
\begin{equation}
A = \mathrm{softmax}\!\left(\frac{QK^\top}{\sqrt{d_k}}\right).
\end{equation}
Using the transformed parameters in \eqref{eq:attn_param_transform} and the identity
\(\mathrm{LN}(X P) = \mathrm{LN}(X) P\),
we obtain
\begin{align}
Q' &= \mathrm{LN}(X P) W_Q' = Q, \\
K' &= \mathrm{LN}(X P) W_K' = K, \\
V' &= \mathrm{LN}(X P) W_V' = V P.
\end{align}
Hence \(A' = A\), and the attention output satisfies
\begin{equation}
\mathrm{SA}'(\mathrm{LN}(X P))
= A' V' W_O'
= (A V W_O) P
= \mathrm{SA}(\mathrm{LN}(X)) P.
\end{equation}

\paragraph{First residual connection.}
Substituting into \eqref{eq:preln_sa},
\begin{equation}
Y' = X P + \mathrm{SA}'(\mathrm{LN}(X P))
= (X + \mathrm{SA}(\mathrm{LN}(X))) P
= Y P.
\end{equation}

\paragraph{Feed-forward network equivariance.}
The FFN is defined as
\begin{equation}
\mathrm{FFN}(x) = W_2\,\sigma(W_1 x),
\end{equation}
where \(\sigma\) is applied elementwise. Using the parameter transformations in
\eqref{eq:ffn_param_transform}, and noting that permutations commute with elementwise nonlinearities, we have
\begin{equation}
\mathrm{FFN}'(\mathrm{LN}(Y P)) = \mathrm{FFN}(\mathrm{LN}(Y))\,P.
\end{equation}

\paragraph{Second residual connection.}
Substituting into \eqref{eq:preln_ffn},
\begin{equation}
Z' = Y' + \mathrm{FFN}'(\mathrm{LN}(Y'))
= (Y + \mathrm{FFN}(\mathrm{LN}(Y))) P
= Z P.
\end{equation}

This establishes that the full Pre-LN Transformer block is equivariant under permutations of the feature dimension.
\end{proof}

\section{Conclusion}
\noindent This paper has provided a systematic study of parametric invariances in compositional models built from alternating polynomial and rectified monomial layers. The invariance group—generated by input-space linear reparameterizations and inter-layer permutation/diagonal scalings—was characterized explicitly, and several applications were developed. For optimization, an invariance-aware geometric program was formulated to minimize regularizers within equivalence classes, yielding improved conditioning without altering the realized function. For quantization and storage, a related geometric program was shown to minimize parameter ranges, reducing dynamic range requirements. Two obfuscation protocols were presented: one enabling private inference where neither party learns the other's secrets, and another enabling privacy-preserving remote training where a cloud server trains on transformed data and parameters without access to the originals. Finally, analogous invariances were identified in self-attention mechanisms—bilinear symmetry in the query–key interaction and linear symmetry in the value–output pathway—along with permutation equivariance of Pre-LN Transformer blocks. Together, these results establish that understanding parametric invariances is not merely a theoretical curiosity but a practical tool for compression, obfuscation, and efficient optimization in modern neural architectures.

\balance

\end{document}